\documentclass[11pt,reqno]{amsart}
\usepackage{amsmath,amssymb,amsthm,color}
\usepackage[tmargin=1in,bmargin=1in,rmargin=1in,lmargin=1in]{geometry}
\usepackage[breaklinks=true]{hyperref}

\newtheorem{theorem}{Theorem}
\newtheorem{prop}[theorem]{Proposition}
\newtheorem{lemma}[theorem]{Lemma}

\theoremstyle{definition}
\newtheorem{defn}[theorem]{Definition}
\newtheorem{remark}[theorem]{Remark}
\newtheorem{question}[theorem]{Question}

\numberwithin{equation}{section}

\begin{document}


\title[Characterizing inner product spaces via strong three-point
homogeneity]{A remark on characterizing inner product spaces via\\
strong three-point homogeneity}

\author{Sujit Sakharam Damase}
\address[S.S.~Damase]{Department of Mathematics, Indian Institute of
Science, Bangalore 560012, India; PhD student}
\email{\tt sujits@iisc.ac.in}

\author{Apoorva Khare}
\address[A.~Khare]{Department of Mathematics, Indian Institute of
Science, Bangalore 560012, India and Analysis \& Probability Research
Group, Bangalore 560012, India}
\email{\tt khare@iisc.ac.in}

\begin{abstract}
We show that a normed linear space is isometrically isomorphic to an
inner product space if and only if it is a strongly $n$-point
homogeneous metric space for any (or every) $n \geqslant 3$.
The counterpart for $n=2$ is the Banach--Mazur problem.
\end{abstract}

\date{\today}



\maketitle

Normed linear spaces and inner product spaces are central to much of
mathematics, in particular analysis and probability. The goal of this
short note is to provide a ``metric'' characterization that we were
unable to find in the literature, of when the norm in a linear space
arises from an inner product. This characterization is classical in
spirit, is in terms of a ``strong'' 3-point homogeneity property that
holds in all inner product spaces, and is adjacent to a well-known open
question.

Two prevalent themes in the early 20th century involved exploring metric
geometry -- e.g.\ when a (finite or) separable metric space isometrically
embeds into the Hilbert space of square-summable sequences $\ell^2 :=
\ell^2(\mathbb{N})$ -- and exploring when a normed linear space
$(\mathbb{B}, \| \cdot \|)$ is Hilbert, i.e.\ $\| \cdot \|^2$ arises from
an inner product on $\mathbb{B}$. See e.g.\
\cite{Birkhoff1}, \cite{Day1}, \cite{Ficken}--\cite{Frechet},
\cite{James1}--\cite{Kakutani}, \cite{Lorch}, \cite{Schoenberg35};
additional works on the latter theme can be found cited in \cite{James2}.
There have also been books -- see e.g.\ \cite{Blumenthal, Dan, Day2,
Istra} -- as well as later works.
This note, while squarely in the latter theme, is strongly
inspired by the former theme -- in which it is worth mentioning the
Mathematics Kolloquium \cite{Menger} of Karl Menger and others in Vienna,
from 1928--36. This long-running lecture series saw contributions in
metric geometry and related areas by Menger, G\"odel, von Neumann, and
others, and led to new developments in metric embeddings, metric
convexity, fixed point theory, and more. The goal of our work is to
provide such a distance-geometric characterization of an inner product.

We begin with some results from the latter theme. Jordan and von Neumann
showed in \cite{JvN} that the norm in a real or complex linear space
$\mathbb{B}$ comes from an inner product if and only if the parallelogram
law holds in $\mathbb{B}$; they also showed the (real and) complex
polarization identity in \textit{loc.\ cit.} We collect this and other
equivalent conditions for the norm $\| \cdot \|$ in a real or complex
linear space $\mathbb{B}$ to arise from an inner product:

\begin{theorem}
Suppose $(\mathbb{B}, \| \cdot \|)$ is a real or complex inner product
space. Its norm arises from an inner product if and only if any of the
following equivalent conditions holds (for the last two, we work only
over $\mathbb{R}$):
\begin{itemize}
\item[(IP1)] (Jordan and von Neumann, \cite{JvN}.)
$\| f + g \|^2 + \| f - g \|^2 = 2 (\| f \|^2 + \| g \|^2)$ for all $f,g
\in \mathbb{B}$.
(The parallelogram law; it originally appears in the
authors' work with Wigner, see \cite[p.~32]{JvNW}.)

\item[(IP2)] (Ficken, \cite{Ficken}.)
If $f, g \in \mathbb{B}$ with $\| f \| = \| g \|$, then $\| \alpha f +
\beta g \| = \| \beta f + \alpha g \|$ for all real $\alpha, \beta$.  

\item[(IP3)] (Day, \cite{Day1}.)
If $f, g \in \mathbb{B}$ with $\| f \| = \| g \| = 1$, then $\| f + g
\|^2 + \| f - g \|^2 = 4$.
(The parallelogram law, but only for rhombi.)

\item[(IP4)] (James, \cite{James2}, when $\dim \mathbb{B}
\geqslant 3$.)
For all $f,g \in \mathbb{B}$, $\| f + \alpha g \| \geqslant \| f \|$ for
all scalars $\alpha$, if and only if $\| g + \alpha f \| \geqslant \| g
\|$ for all scalars $\alpha$.
(Symmetry of Birkhoff--James orthogonality.)

\item[(IP5)] (Lorch, \cite{Lorch}, over $\mathbb{R}$.)
There exists a fixed constant $\gamma' \in \mathbb{R} \setminus \{ 0, \pm
1 \}$\footnote{We correct a small typo in Lorch's (IP5) in \cite{Lorch}: he
stated $\gamma' \neq 0, 1$ but omitted excluding $-1$; but clearly
$\gamma' = -1$ ``works'' for every normed linear space $\mathbb{B}$ and
all vectors $f',g' \in \mathbb{B}$, since $\| f' - g' \| = \| g' - f'
\|$.}
such that whenever $f',g' \in \mathbb{B}$ with $\| f' \| = \| g' \|$, we
have $\| f' + \gamma' g' \| = \| g' + \gamma' f' \|$.

\item[(IP6)] (Lorch, \cite{Lorch}, over $\mathbb{R}$.)
There exists a fixed constant $\gamma \in \mathbb{R} \setminus \{ 0, \pm
1 \}$ such that whenever $f,g \in \mathbb{B}$ with $\| f + g \| = \| f -
g \|$, we have $\| f + \gamma g \| = \| f - \gamma g \|$.
\end{itemize}
\end{theorem}

Indeed, that every inner product space satisfies these properties is
immediate, while the implications (IP5) $\Longleftrightarrow$ (IP6)
$\implies$ (IP2) were shown by Lorch; and that (IP1), (IP2), (IP3), (IP4)
imply that $\| \cdot \|^2$ arises from an inner product were shown by the
respectively named authors above.

All of these characterizations of an inner product use the norm and the
vector space structure (over $\mathbb{R}$ or $\mathbb{C}$) on
$\mathbb{B}$. The goal of this note is to isolate the inner product using
the metric in $\mathbb{B}$ but \textit{avoiding} both the additive
structure and the (real or complex) scalar multiplication.
Thus, our result is in the spirit of both of the aforementioned classical
themes: characterizing inner products in normed linear spaces, while
using metric geometry alone.

\section{The main result and its proof}\label{S1}

To state our result, first recall from \cite{Wang} or even
\cite[p.~470]{Birkhoff2} that for an integer $n \geqslant 1$, a metric
space $(X,d)$ is \textit{$n$-point homogeneous} if given two finite
subsets $Y,Y' \subseteq X$ with $|Y| = |Y'| \leqslant n$, any isometry $T
: Y \to Y'$ extends to an isometry $: X \to X$. We will require a
somewhat more restrictive notion:

\begin{defn}
A metric space $(X,d)$ will be termed \textit{strongly $n$-point
homogeneous} if given subsets $Y,Y' \subseteq X$ with $|Y| = |Y'|
\leqslant n$, each isometry $T : Y \to Y'$ can be extended to an onto
isometry $: X \twoheadrightarrow X$.
\end{defn}

We now motivate our main result (and the above definition via onto
isometries). It seems to be folklore that the Euclidean space
$\mathbb{R}^k$ is $n$-point homogeneous for all $n$ -- and more strongly,
satisfies that every isometry between finite subsets $Y,Y' \subseteq
\mathbb{R}^k$, upon pre- and post- composing with suitable translations
in order to send $0$ to $0$, extends to an orthogonal linear map $:
\mathbb{R}^k \to \mathbb{R}^k$.

\begin{remark}
While not central to our main result, we will explain below why this
property also holds for infinite subsets of $(\mathbb{R}^k, \| \cdot
\|_2)$; a weakening of it was termed the \textit{free mobility postulate}
by Birkhoff \cite[pp.\ 469--470]{Birkhoff2}. However, this postulate is
not satisfied by any infinite-dimensional Hilbert space for infinite
subsets. This was already pointed out in 1944 by Birkhoff in
\textit{loc.\ cit.}; for a specific counterexample, see e.g.\ the proof
of \cite[Theorem 11.4]{WW} in $\ell^2$, where the left-shift operator
sending $Y := \{ (x_n)_{n \geqslant 0} \in \ell^2 : x_1 = 0 \}$ onto $Y'
:= \ell^2$ is an isometry that does not extend to $(1,0,0,\dots)$.
\end{remark}

Returning to our motivation: in fact all inner product spaces satisfy
this ``orthogonal extension property'' for all pairs of isometric finite
subsets $Y,Y'$, as we explain below. Here we are interested in the
converse question, i.e.,
\begin{enumerate}
\item[(a)] if this ``orthogonal extension property'' (with $Y,Y'$ finite)
characterizes inner product spaces (among normed linear spaces); and

\item[(b)] if yes, then how much can this property be weakened without
disturbing the characterization -- and if it can in fact be weakened to
use metric geometry alone.
(In an arbitrary normed linear space, we necessarily cannot use
orthogonality or inner products; but we also want to not use the vector
space operations either.)
\end{enumerate}

This note shows that indeed (a)~holds. Moreover,
(b)~we can indeed weaken the orthogonal extension property to
(i)~replacing the orthogonal linear map by merely an onto isometry -- not
necessarily linear \textit{a priori} -- and (ii)~working with 3-point
subsets $Y,Y'$. More precisely:

\begin{theorem}\label{Tmain}
Suppose $(\mathbb{B}, \| \cdot \|)$ is a nonzero real or complex normed
linear space. Then $\| \cdot \|^2$ arises from an inner product -- real
or complex, respectively -- if and only if $\mathbb{B}$ is strongly
$n$-point homogeneous for any (equivalently, every) $n \geqslant 3$.
\end{theorem}

\begin{remark}
Below, we will provide additional equivalent -- and \textit{a priori}
weaker -- conditions to add to Theorem~\ref{Tmain} in characterizing an
inner product. See Theorems~\ref{T4} and~\ref{T5}.
\end{remark}

To the best of our ability -- and that of a dozen experts -- we were
unable to find such a result proved in the literature. Before proceeding
to its proof, we discuss the assertion for $n=1,2$. If $n=1$ then
Theorem~\ref{Tmain} fails to hold, since every normed linear space
$\mathbb{B}$ is strongly one-point homogeneous: given $x, y \in
\mathbb{B}$, the translation $z \mapsto z + y-x$ is an onto isometry
sending $x$ to $y$.

If instead $n=2$ then one is asking if there exists a (real) linear space
$(\mathbb{B}, \| \cdot \|)$ with $\| \cdot \|^2$ not arising from an
inner product, such that given any $f,f',g,g' \in \mathbb{B}$ with $\|
f'-f \| = \| g'-g \|$, every isometry sending $f,f'$ to $g,g'$
respectively extends to an onto isometry of $\mathbb{B}$. By pre- and
post- composing with translations, one can assume $f=g=0$ and $\| f' \| =
\| g' \|$; now one is asking if every real normed linear space with
transitive group of onto-isometries fixing $0$ (these are called
\textit{rotations}) is isometrically isomorphic to an inner product
space. Thus we come to the well-known \textit{Banach--Mazur
problem}~\cite{Banach} --
which was affirmatively answered by Mazur for finite-dimensional
$\mathbb{B}$~\cite{Mazur},
has counterexamples among non-separable $\mathbb{B}$~\cite{PR},
and remains open for infinite-dimensional separable $\mathbb{B}$.
(This is also called the \textit{Mazur rotations problem}; see the
recent survey~\cite{survey}.)
This is when $n=2$; and the $n \geqslant 3$ case is Theorem~\ref{Tmain}.

\begin{remark}
Given the preceding paragraph, one can assume in Theorem~\ref{Tmain} that
$\mathbb{B}$ is infinite-dimensional, since if $\dim \mathbb{B} < \infty$
then strong 3-point homogeneity implies strong 2-point homogeneity, which
by Mazur's solution \cite{Mazur} to the Banach--Mazur problem implies
$\mathbb{B}$ is Euclidean. That said, our proof of Theorem~\ref{Tmain}
works uniformly over all normed linear spaces, and we believe is simpler
than using the Banach--Mazur problem (which moreover cannot be applied
for all $\mathbb{B}$).
\end{remark}

\begin{proof}[Proof of Theorem~\ref{Tmain}]
We begin by proving the real case.
We first explain the forward implication, starting with $\mathbb{B} \cong
(\mathbb{R}^k, \| \cdot \|_2)$ for an integer $k \geqslant 1$. As
is asserted (without proof) on \cite[p.~470]{Birkhoff2}, $\mathbb{R}^k$
is $n$-point homogeneous for every $n \geqslant 1$; we now show that it
is moreover strongly $n$-point homogeneous -- in fact, that it satisfies
the ``orthogonal extension property'' above:\smallskip

\textit{Let $Y,Y' \subseteq \mathbb{R}^k$, and let $T : Y \to Y'$ be an
isometry. By translating $Y$ and $Y'$, assume $0 \in Y \cap Y'$ and $T(0)
= 0$. Then $T$ extends to an orthogonal self-isometry $\widetilde{T}$ of
$(\mathbb{R}^k, \| \cdot \|_2)$.}\smallskip

This claim can be proved from first principles and is likely folklore
(e.g.\ see a skeleton argument in \cite[Section~38]{Blumenthal}).
However, for self-completeness, we present a detailed sketch via a
``lurking isometry'' argument, along with some supplementary remarks. For
full details, see \cite[Theorem 22.3]{Khare} and its proof.

The first step is to note that vectors $y_0, \dots, y_n$ in an inner
product space are linearly dependent if and only if their Gram matrix $G
:= ( \langle y_i, y_j \rangle)_{i,j=0}^n$ is singular, since
\begin{equation}\label{Eobserve}
v := \sum_{j=0}^n c_j y_j = {\bf 0}
\quad \implies \quad G {\bf c} = {\bf 0}
\quad \implies \quad {\bf c}^* G {\bf c} = 0
\quad \implies \quad \| v \|^2 = 0.
\end{equation}

This simple fact yields several noteworthy consequences; we mention two
here, without proof. The first is an 1841 result by Cayley \cite{Cayley}
(during his undergraduate days):

\begin{lemma}
Let $(X = \{ x_0, \dots, x_n \}, d)$ be a finite metric space, and $\Psi
: X \to \ell^2$ an isometry. Then the vectors $\Psi(X)$ are affine
linearly-dependent (i.e.\ lie in an $(n-1)$-dimensional affine subspace)
if and only if the Cayley--Menger matrix of $X$ is singular:
\begin{equation}
CM(X)_{(n+2) \times (n+2)} := \begin{pmatrix}
0 & d_{01}^2 & d_{02}^2 & \cdots & d_{0n}^2 & 1\\
d_{10}^2 & 0 & d_{12}^2 & \cdots & d_{1n}^2 & 1\\
d_{20}^2 & d_{21}^2 & 0 & \cdots & d_{2n}^2 & 1\\
\vdots & \vdots & \vdots & \ddots & \vdots & \vdots\\
d_{n0}^2 & d_{n1}^2 & d_{n2}^2 & \cdots & 0 & 1\\
1 & 1 & 1 & \cdots & 1 & 0
\end{pmatrix}, \ \text{where} \ d_{ij} = d_X(x_i, x_j).
\end{equation}
\end{lemma}

The second consequence is used below, but also underlies the Global
Positioning System (GPS) ``trilateration'' -- i.e.\ that every point on
say a Euclidean plane $P$ is uniquely determined by its distances from
three non-collinear points in $P$. More generally:

\begin{prop}\label{PGPS}
Fix $y_0 = 0, \dots, y_n \in \ell^2$. The following are equivalent for $y
\in \ell^2$:
\begin{enumerate}
\item $y$ is (uniquely) determined by its distances from $y_0, \dots, y_n$.

\item $y$ is in the span of $y_1, \dots, y_n$.
\end{enumerate}
\end{prop}

\noindent (See e.g.\ \cite[Proposition~22.7]{Khare}.)
Returning to the proof of the orthogonal extension property for
$\mathbb{R}^k$: fix a maximal linearly independent subset $\{ y_1,
\dots, y_r \}$ in $Y$. Then we claim, so is $\{ T(y_1), \dots, T(y_r)
\}$. Indeed, by polarization we have
\begin{align}\label{Einnerprod}
2 \langle y_i, y_j \rangle
= &\ \| y_i - 0 \|_2^2 + \| y_j - 0 \|_2^2 - \| y_i - y_j \|_2^2\\
= &\ \| T(y_i) - T(0) \|_2^2 + \| T(y_j) - T(0) \|_2^2 -
\| T(y_i) - T(y_j) \|_2^2 = 2 \langle T(y_i), T(y_j) \rangle\notag
\end{align}
for $1 \leqslant i,j \leqslant r$. Thus the Gram matrix of the $T(y_i)$
equals that of the $y_i$, and hence is invertible as well; so
by~\eqref{Eobserve} the $T(y_j)$ are linearly independent too. Moreover,
for any other $y \in Y$, the Gram matrix of $y, y_1, \dots, y_r$ is
singular by~\eqref{Eobserve}, hence so is the Gram matrix of $T(y),
T(y_1), \dots, T(y_r)$ by~\eqref{Einnerprod}.

Next, we claim that the linear extension $\widetilde{T}$ of $T$ from $\{
y_1 \dots, y_r \}$ to ${\rm span}_{\mathbb{R}}(Y)$ (hence mapping into
${\rm span}_{\mathbb{R}}(Y')$) agrees with $T$ on $Y$.
Indeed, by maximality one writes $y \in {\rm span}_{\mathbb{R}}(Y)$ as
$\sum_{j=1}^r c_j(y) y_j$ and $T(y) \in {\rm
span}_{\mathbb{R}}(Y')$ as $\sum_{j=1}^r c'_j(y) T(y_j)$;
then~\eqref{Einnerprod} and Proposition~\ref{PGPS} show that $c_j \equiv
c'_j$ on $Y$ (for all $j$). Hence $\widetilde{T} \equiv T$ on $Y$.

It also follows from~\eqref{Einnerprod} that $\widetilde{T}$ preserves
lengths, hence is injective. Finally, choose orthonormal bases of the
orthocomplements in $\mathbb{R}^k$ of ${\rm span}_{\mathbb{R}}(Y)$
and of $\widetilde{T} \left( {\rm span}_{\mathbb{R}}(Y) \right)$, and map
the first of these bases (within ${\rm span}_{\mathbb{R}}(Y)^\perp$)
bijectively onto the second; then extend $\widetilde{T}$ to all of
$\mathbb{R}^k$ by linearity. Direct-summing these two orthogonal linear
maps yields the desired extension of $T$ to a linear self-isometry
$\widetilde{T}$ of $(\mathbb{R}^k, \| \cdot \|_2)$.

This linear orthogonal map $\widetilde{T}$ is necessarily injective on
$\mathbb{R}^k$, hence surjective as well. This shows the orthogonal
extension property for all (possibly infinite) isometric subsets $Y,Y'
\subseteq \mathbb{R}^k$, and hence the forward implication in the main
result for $\mathbb{B} \cong \mathbb{R}^k$.

If instead $\mathbb{B}$ is an infinite-dimensional inner product space,
and $T$ an isometry between $Y, Y' \subset \mathbb{B}$ of common size at
most $n$, first pre- and post- compose by translations to assume $0
\in  Y \cap Y'$ and $T(0) = 0$. Now let $\mathbb{B}_0 \cong
(\mathbb{R}^k, \| \cdot \|_2)$ be the span of $Y \cup Y'$. By the above
analysis, $T : Y \to Y'$ extends to an orthogonal operator on
$\mathbb{B}_0$, which we still denote by $T$; as $\mathbb{B}_0$ is a
complete subspace of $\mathbb{B}$, by the ``projection theorem'' (or from
first principles) we get $\mathbb{B} = \mathbb{B}_0 \oplus
\mathbb{B}_0^\perp$. Now the bijective orthogonal map $T|_{\mathbb{B}_0}
\oplus {\rm Id}|_{\mathbb{B}_0^\perp}$ completes the proof of the forward
implication -- for any $n \geqslant 1$.

We next come to the reverse implication; now $n \geqslant 3$. From the
definitions, it suffices to work with $n=3$. Moreover, the cases of $\dim
\mathbb{B} = 0,1$ are trivial since $\mathbb{B}$ is then an inner product
space, so the reader may also assume $\dim \mathbb{B} \in [2,\infty]$ in
the sequel, if required.

We work via contradiction: let $(\mathbb{B}, \| \cdot \|)$ be a nonzero
3-point homogeneous normed linear space, with $\| \cdot \|^2$ not induced
by an inner product. Then Lorch's condition~(IP5) fails to hold for any
$\gamma' \in \mathbb{R} \setminus \{ 0, \pm 1 \}$. Fix such a scalar
$\gamma'$; then there exist vectors $f',g' \in \mathbb{B}$ with $\| f' \|
= \| g' \|$ but $\| f' + \gamma' g' \| \neq \| g' + \gamma' f' \|$. In
particular, $0, f', g'$ are distinct.
Now let $Y = Y' = \{ 0, f', g' \}$, and consider the isometry $T : Y \to
Y'$ which fixes $0$ and interchanges $f',g'$. By hypothesis, $T$ extends
to an onto isometry $: \mathbb{B} \to \mathbb{B}$, which we also denote
by $T$. But then $T$ is affine-linear by the Mazur--Ulam theorem
\cite{MU}, hence is a linear isometry as $T(0) = 0$. This yields
\[
\| f' + \gamma' g' \| = \| f' - (- \gamma') g' \| = \| T(f') - T(-
\gamma' g') \| = \| g' + \gamma' f' \|,
\]
which provides the desired contradiction, and proves the reverse
implication.

\begin{remark}
An alternate argument to the one provided above (for the reverse
implication) goes as follows.
\textit{By the parallelogram law (IP1), it suffices to show every plane
$P \subseteq \mathbb{B}$ is Euclidean. Suppose $0 \neq f,g \in P$ with
$\| f \| = \| g \|$; then the isometry of $\{ 0, f, g \}$ that fixes $0$
and exchanges $f,g$ extends to an onto isometry of $\mathbb{B}$. By
Mazur--Ulam \cite{MU}, $T$ is affine-linear (hence linear as $T(0) = 0$),
and thus sends $P$ onto itself. But then $\| \alpha f + \beta g \| = \|
T(\alpha f + \beta g) \| = \| \beta f + \alpha g \|$ for all $\alpha,
\beta \in \mathbb{R}$. By Ficken's result (IP2) for $P$, $P$ is
Euclidean.}
We note that this argument is somewhat more involved than the one above,
as it makes use of
(i)~Ficken's characterization (IP2), which as Lorch wrote \cite{Lorch} is
\textit{a priori} more involved than Lorch's (IP5); and
(ii)~the parallelogram law (IP1), which our argument does not use (nor did
Lorch).\footnote{One can also see \cite[(2.8)]{Dan}, where Dan mentions
the flip map of an isosceles triangle. However, Dan requires the isometry
$T$ to be \textit{linear}, which is strictly stronger than our
characterization-hypothesis of strong homogeneity. Moreover, this
linearity assumption is indeed required by Dan, since he mentions in the
very next line that linearity can be dropped if $\mathbb{B}$ is assumed
to be strictly convex (see Theorem~\ref{Tstrictlyconvex}). In contrast,
Theorem~\ref{Tmain} makes no assumption either about the isometry $T$
(other than surjectivity), nor about the normed linear space
$\mathbb{B}$.}
\end{remark}

The above analysis proves the real case; now suppose $(\mathbb{B}, \|
\cdot \|)$ is a nonzero complex normed linear space. The forward
implication follows from the real case, since $(\mathbb{B}, \| \cdot \|)$
is also a real normed space -- which we denote by
$\mathbb{B}|_{\mathbb{R}}$. For the same reason, in the reverse direction
we obtain that $\mathbb{B}|_{\mathbb{R}}$ is isometrically isomorphic to
a real inner product space: $\| f \| = \sqrt{ \langle f, f
\rangle_{\mathbb{R}} }$ for a real inner product $\langle \cdot, \cdot
\rangle_{\mathbb{R}}$ on $\mathbb{B}|_{\mathbb{R}}$ and all $f \in
\mathbb{B}|_{\mathbb{R}}$. Now the complex polarization trick of
Jordan--von Neumann \cite{JvN} gives that
\begin{equation}\label{Epolarize}
\langle f, g \rangle := \langle f, g \rangle_{\mathbb{R}}
- i \langle if, g \rangle_{\mathbb{R}}
\end{equation}
is indeed a complex inner product on $\mathbb{B}$ satisfying: $\| f \| =
\sqrt{\langle f, f \rangle}$ for all $f$.
\end{proof}

\section{A second proof; Lorch's characterizations for complex linear
spaces}

We next provide a second proof of the ``reverse implication'' of
Theorem~\ref{Tmain} for complex linear spaces $\mathbb{B} =
\mathbb{B}|_{\mathbb{C}}$, which requires the \textit{complex version} of
Lorch's condition~(IP5) above. Note that the above argument over
$\mathbb{R}$ cannot immediately proceed verbatim over $\mathbb{C}$ for
two reasons:
\begin{enumerate}
\item[(a)] Lorch's condition~(IP5) needs to be verified as characterizing a
complex inner product.

\item[(b)] The Mazur--Ulam theorem does not go through over $\mathbb{C}$
-- e.g., the isometry $\eta : (z,w) \mapsto (\overline{z},w)$ of the
Hilbert space $(\mathbb{C}^2, \| \cdot \|_2)$ sends $(0,0)$ to itself,
but is neither $\mathbb{C}$-linear nor $\mathbb{C}$-antilinear.
\end{enumerate}

However, since $\eta$ is $\mathbb{R}$-linear on $\mathbb{C}^2$, this
reveals how to potentially fix~(b) for a second proof of the reverse
implication: it suffices to use Lorch's condition~(IP5) for $\gamma'$
still real -- and avoiding $0, \pm 1$ as above -- now for all vectors
$f',g' \in \mathbb{B}|_{\mathbb{C}}$. If this still characterizes a
complex inner product, then one could continue the above proof verbatim,
using the Mazur--Ulam theorem for \textit{real} normed linear spaces and
replacing the word ``linear'' twice by ``$\mathbb{R}$-linear'' in the
proof.

Thus we need to verify if Lorch's condition~(IP5) characterizes a complex
inner product. More broadly, one can ask which of Lorch's conditions
$(I_1)$--$(I_6)$ and $(I'_1) =$ (IP5) in \cite{Lorch} -- which
characterized an inner product in a \textit{real} normed linear space --
now do the same over $\mathbb{C}$:

\begin{theorem}[Lorch, \cite{Lorch}]
Suppose $\mathbb{B}$ is a real normed linear space. Then the norm
is induced by an inner product if and only if any of the following
equivalent conditions hold:
\begin{itemize}
\item[$(I_1)$] {\rm (Stated above as (IP6).)}
There exists a fixed constant $\gamma \in \mathbb{R} \setminus \{ 0, \pm
1 \}$ such that whenever $f,g \in \mathbb{B}$ with $\| f + g \| = \| f -
g \|$, we have $\| f + \gamma g \| = \| f - \gamma g \|$.

\item[$(I'_1)$] {\rm (Stated above as (IP5).)}
There exists a fixed constant $\gamma' \in \mathbb{R} \setminus \{ 0, \pm
1 \}$ such that whenever $f',g' \in \mathbb{B}$ with $\| f' \| = \| g'
\|$, we have $\| f' + \gamma' g' \| = \| g' + \gamma' f' \|$.

\item[$(I_2)$] A triangle is isosceles if and only if two medians are
equal:
\[
f+g+h = 0, \ \| f \| = \| g \| \quad \implies \quad \| f-h \| = \| g-h \|,
\qquad \forall f,g,h \in \mathbb{B}.
\]

\item[$(I_3)$] For all $f,g,h,k \in \mathbb{B}$,
\[
f+g+h+k = 0, \ \|f\| = \|g\|, \ \|h\| = \|k\| \quad \implies \quad
\|f - h\| = \|g - k\| \text{ and } \|g - h\| = \|f - k\|.
\]

\item[$(I_4)$] For all $f_1, f_2 \in \mathbb{B}$, the expression
\[
\phi(f_1, f_2; g) = \|f_1 + f_2 + g\|^2 + \|f_1 + f_2 - g\|^2 - \|f_1 -
f_2 - g\|^2 - \|f_1 - f_2 + g\|^2
\]
is independent of $g \in \mathbb{B}$.

\item[$(I_5)$] If $f,g \in \mathbb{B}$ and $\|f\| = \|g\|$, then
$\| \alpha f + \alpha^{-1} g \| \geqslant \| f+g \|$ for all $0 \neq
\alpha \in \mathbb{R}$.

\item[$(I_6)$] For a fixed integer $n \geqslant 3$, and vectors $f_1,
\ldots, f_n \in \mathbb{B}$, we have
\[
f_1 + \cdots + f_n = 0 \quad \implies \quad
\sum_{i < j} \|f_i - f_j\|^2 = 2n \sum_{i=1}^n \|f_i\|^2.
\]
\end{itemize}
\end{theorem}

In fact, these characterizations have since been cited and applied in
many papers that work over real linear spaces. The need to work over
complex normed linear spaces -- as well as the fact that several
characterizations in \cite{Day1,Ficken,JvN} before Lorch held uniformly
over both $\mathbb{R}$ and $\mathbb{C}$ -- provide natural additional
reasons to ask if Lorch's conditions also characterize complex inner
products.

We quickly explain why this does hold. Indeed, Lorch's characterizations
themselves involve real scalars, so even if one starts with a complex
normed linear space $\mathbb{B}$, one still obtains a real inner product
$\langle \cdot, \cdot \rangle_{\mathbb{R}}$ on $\mathbb{B} =
\mathbb{B}|_{\mathbb{R}}$. Now the discussion around~\eqref{Epolarize}
recovers the complex inner product from $\langle \cdot, \cdot
\rangle_{\mathbb{R}}$. This yields

\begin{theorem}[``Complex Lorch'']
A complex nonzero normed linear space $(\mathbb{B}, \| \cdot \|)$ is
isometrically isomorphic to a complex inner product space if and only if
any of the following conditions of Lorch \cite{Lorch} holds:
$(I_1) =$ (IP6), $(I'_1) =$ (IP5),
or $(I_2), \dots, (I_6)$ -- now stated verbatim in $\mathbb{B}$, with the
constants $\gamma, \gamma', \alpha$ still being real.
\end{theorem}

\noindent In particular, this provides a second proof of one implication
in Theorem~\ref{Tmain} over $\mathbb{C}$.

\section{Isosceles triangle characterization; weaker notions of homogeneity}

Note that the arguments in Section~\ref{S1} in fact show a strengthening
of Theorem~\ref{Tmain}. Namely: continuing the discussion preceding
Theorem~\ref{Tmain}, the ``orthogonal extension property'' can be
weakened to strong 3-point homogeneity, and even weaker -- wherein one
only works with $Y = Y'$ the vertices of an isosceles triangle in
$\mathbb{B}$ (with the two equal sides of specified length).

\begin{theorem}\label{T4}
The following are equivalent for a nonzero real/complex normed linear
space $(\mathbb{B}, \| \cdot \|)$.
\begin{enumerate}
\item $\| \cdot \|^2$ arises from an (real or complex) inner product on
$\mathbb{B}$.

\item If $Y,Y' \subset \mathbb{B}$ are finite subsets, with $|Y| = |Y'|$,
then any isometry $: Y \to Y'$ -- up to pre- and post- composing by
translations -- extends to an $\mathbb{R}$-linear onto isometry
$: \mathbb{B} \to \mathbb{B}$.

\item $\mathbb{B}$ is strongly $n$-point homogeneous for any
(equivalently, every) $n \geqslant 3$.

\item Given an isosceles triangle in $\mathbb{B}$ with vertex set $Y = Y'
=  \{ 0, f', g' \}$ such that $\| f' \| = \| g' \| = 1$, the isometry $:
Y \to Y'$ that fixes $0$ and flips $f',g'$ extends to an onto isometry of
$\mathbb{B}$.
\end{enumerate}
\end{theorem}

Indeed, we showed above that $(4) \implies (1) \implies (2)$ (as the
$f',g'$ in Lorch's condition~(IP5) can be simultaneously rescaled), while
$(2) \implies (3) \implies (4)$ is trivial. To see why one cannot assert
$\mathbb{C}$-linearity in the second statement, let
\[
\mathbb{B} = (\mathbb{C}^2, \| \cdot \|_2), \qquad
Y = \{ (0,0), (1,0), (i,0) \}, \qquad
Y' = \{ (0,0), (1,0), (0,1) \},
\]
and let $T : Y \to Y'$ fix $(0,0)$ and $(1,0)$, and send $(i,0)$ to
$(0,1)$. Then $T$ necessarily cannot extend to a $\mathbb{C}$-linear
map.
Also note that akin to Theorem~\ref{Tmain}, here too the final three
assertions each characterize inner products using (isosceles) metric
geometry alone, and without appealing to the vector space structure in
$\mathbb{B}$.

We conclude by exploring two weakenings of the notion of ``strong''
$n$-point homogeneity that was used to characterize when $\mathbb{B}$ is
an inner product space. The first question is to ask if one can work with
``usual'' $n$-point homogeneity, i.e.\ if one can remove the ``onto''
part of that definition. Note that in the literature, the requirement of
extending subset-isometries to onto-isometries of the whole (metric)
space is perhaps more natural than merely into-isometries, given their
appearance in well-known results and open problems in the Banach space
literature -- the Mazur--Ulam theorem \cite{MU}, the Banach--Mazur
problem \cite{Banach,Mazur}, and Tingley's problem \cite{Tingley} among
others -- and also given the folklore fact that this holds for all inner
product spaces \textit{and} in our results above. Now having studied the
``onto'' picture in detail, the first question (above) is natural.

To answer it, note that the proof of Theorem~\ref{Tmain} required the
``onto'' hypothesis solely to invoke the Mazur--Ulam theorem. This
hypothesis can thus be bypassed, at a cost:

\begin{theorem}\label{Tstrictlyconvex}
Suppose $(\mathbb{B}, \| \cdot \|)$ is a nonzero real or
complex normed linear space that is moreover strictly convex:
\[
\| a+b \| = \|a\| + \|b\| \quad \implies \quad  a,b \in \mathbb{B}
\text{ are linearly dependent}.
\]
Then $\| \cdot \|^2$ arises from a (real or complex, respectively) inner
product if and only if $\mathbb{B}$ is $n$-point homogeneous for any
(every) $n \geqslant 3$ -- equivalently, if the flip map on the vertices
of each isosceles triangle extends to a self-isometry of $\mathbb{B}$.
\end{theorem}

In particular, this reveals yet another equivalent condition to the inner
product property: $\mathbb{B}$ is strictly convex and (``usual'')
$3$-point homogeneous.

\begin{proof}
The reverse implication over $\mathbb{R}$ follows from the corresponding
part of Theorem~\ref{Tmain}, using Baker's result~\cite{Baker} that the
Mazur--Ulam theorem does not require the surjectivity hypothesis if
$\mathbb{B}$ is strictly convex.
The reverse implication over $\mathbb{C}$ is now shown as in
Theorem~\ref{Tmain};
and the forward implication follows from Theorem~\ref{Tmain}.
\end{proof}

We formulate the natural question corresponding to the remaining case.

\begin{question}
Suppose $\mathbb{B}$ is a (real) normed linear space that is not strictly
convex (in particular, not strongly $3$-point homogeneous or equivalently
an inner product space). Can $\mathbb{B}$ be $3$-point homogeneous?
(To start, one can assume $\mathbb{B}$ separable and complete.)
\end{question}

We end by answering a natural question that arises from
Theorem~\ref{T4}(4). Note that (strong) $n$-point homogeneity concerns
subsets $Y,Y'$ of size at most $n$. However, Theorem~\ref{T4}(4) shows
that one can work with subsets of size exactly $n=3$. Given
Theorem~\ref{Tmain}, one can ask if subsets of size exactly $4$, or any
fixed integer $>4$, will suffice to isolate an inner product. Our final
result provides an affirmative answer, thereby adding to
Theorem~\ref{T4}:

\begin{theorem}\label{T5}
The conditions in Theorem~\ref{T4} are further equivalent to:
\begin{enumerate}
\setcounter{enumi}{4}
\item For any fixed $3 \leqslant n < \infty$, every isometry $T : Y \to
Y'$ of $n$-point subsets $Y,Y' \subseteq X$ extends to an onto isometry
of $\mathbb{B}$.
\end{enumerate}
\end{theorem}

\begin{proof}
Clearly $(2) \implies (5)$. Conversely, it suffices to assume $(5)$ and
show Theorem~\ref{T4}(1). We follow the proof of Theorem~\ref{Tmain}: if
$\| \cdot \|$ is not induced by an inner product, there exist $\gamma'
\in \mathbb{R} \setminus \{ 0, \pm 1 \}$ and $f',g' \in \mathbb{B}$ with
$\| f' \| = \| g' \|$, but $\| f' + \gamma' g' \| \neq \| g' + \gamma' f'
\|$. Thus, $0, f', g'$ are distinct.

There are now two cases. First if $g' = -f'$ then the flip map extends to
the onto isometry $T \equiv - {\rm id}_{\mathbb{B}}$. Else $f', g'$ are
$\mathbb{R}$-linearly independent, so their $\mathbb{R}$-span is a real
normed plane $P \subseteq \mathbb{B}$, say. As all norms on
$\mathbb{R}^2$ are equivalent, $(P, \| \cdot \|)$ is linearly
bi-Lipschitz -- hence homeomorphic -- to $(\mathbb{R}^2, \| \cdot \|_2)$.

We now claim that the ``equidistant locus in $P$'' -- i.e., the locus
$Z_{f',g'} := \{ h \in P : \| h - f' \| = \| h - g' \| \}$ is
uncountable. Indeed, if this is not true then $Z_{f',g'}$ is at most
countable, so $P \setminus Z_{f',g'}$ is homeomorphic to a
countably-punctured Euclidean plane, hence is path-connected. As the
function
\[
\varphi : P \to \mathbb{R}; \quad h \mapsto \| h - f' \| - \| h - g' \|
\]
is continuous on $P \setminus Z_{f',g'}$, and switches signs from $f'$ to
$g'$, it must vanish at an ``intermediate'' point in $P \setminus
Z_{f',g'}$. This contradiction shows the claim.

Finally, as $n \geqslant 3$, choose distinct points $h_1, \dots, h_{n-3}
\in Z_{f',g'} \setminus \{ 0 \}$, and define $Y, Y'$ via:
\[
Y = Y' := \{ 0, f', g', \ h_1, \dots, h_{n-3} \}.
\]

Let $T : Y \to Y'$ interchange $f',g'$ and fix all other points in $Y$.
By hypothesis, $T$ extends to an onto isometry of $\mathbb{B}$. Now
repeating the proof of Theorem~\ref{Tmain}, it follows that $\mathbb{B}$
is a real inner product space. Finally, if $\mathbb{B}$ was a complex
linear space then the discussion around~\eqref{Epolarize} reveals the
complex inner product structure.
\end{proof}

\subsection*{Acknowledgments}

We thank Pradipta Bandopadhyay, Deepak Gothwal, Lajos Moln\'ar, and Amin
Sofi for valuable discussions.
We also thank both anonymous referees for their detailed reading of our
paper. In particular, we are grateful to one referee for numerous
polishings that distinctly improved our exposition, and the other referee
for explaining why Theorem~\ref{T5} holds.
A.K.~was partially supported by SwarnaJayanti Fellowship grants
SB/SJF/2019-20/14 and DST/SJF/MS/2019/3 from SERB and DST (Govt.~of
India), a Shanti Swarup Bhatnagar Award from CSIR (Govt.\ of India), and
the DST FIST program 2021 [TPN--700661].



\end{document}